\documentclass[11pt]{article}%
\usepackage{amsmath}
\usepackage{amsfonts}
\usepackage{amssymb}
\usepackage{graphicx}%
\newtheorem{theorem}{Theorem}

\newtheorem{corollary}[theorem]{Corollary}

\newtheorem{lemma}[theorem]{Lemma}

\newtheorem{proposition}[theorem]{Proposition}
\newtheorem{remark}[theorem]{Remark}

\newenvironment{proof}[1][Proof]{\noindent\textbf{#1.} }{\ \rule{0.5em}{0.5em}}
\graphicspath{    {converted_graphics/}    {/}}
\begin{document}

\begin{center}
{\Large Extinction, periodicity and multistability in a }

{\Large Ricker Model of Stage-Structured Populations }

\medskip

{N. LAZARYAN and H. SEDAGHAT\footnote{Department of Mathematics, Virginia
Commonwealth University Richmond, Virginia, 23284-2014, USA; Email:
hsedagha@vcu.edu, lazaryans@vcu.edu}}
\end{center}

\medskip

\begin{abstract}
We study the dynamics of a second-order difference equation that is derived
from a planar Ricker model of two-stage biological populations. We obtain
sufficient conditions for global convergence to zero in the non-autonomous
case. This gives general conditions for extinction in the biological context.
We also study the dynamics of an autonomous special
case of the equation that generates multistable periodic and non-periodic
orbits in the positive quadrant of the plane.

\end{abstract}

\medskip

\section{Introduction}

Planar systems of type%
\begin{align}
x_{n+1}  &  =\sigma_{1,n}y_{n}+\sigma_{2,n}x_{n}\label{pc1}\\
y_{n+1}  &  =\beta_{n}x_{n}e^{\alpha_{n}-c_{1,n}x_{n}-c_{2,n}y_{n}}
\label{pc2}%
\end{align}
where $\alpha_{n},\beta_{n},\sigma_{i,n},c_{i,n}$ are non-negative numbers for
$i=1,2$ and $n\geq0$ have been used to model single-species, two-stage
populations (e.g. juvenile and adult); see \cite{Cush}--\cite{GL}, \cite{LP}
and \cite{Z}. The exponential function that defines the time and density
dependent fertility rate classifies the above system as a Ricker model. The
coefficients $\sigma_{i,n}$ are typically composed of the natural survival
rates $s_{i}$ and possibly other factors. For example, they may include
harvesting parameters, as in \cite{LP} and \cite{Z}:
\begin{equation}
\sigma_{i}=(1-h_{i})s_{i},\quad\beta=(1-h_{1})b,\quad c_{1}=(1-h_{1}%
)\gamma,\quad c_{2}=0 \label{C0}%
\end{equation}

All parameters in (\ref{C0}) are assumed to be independent of $n$. In this
case, $h_{i},s_{i}\in\lbrack0,1]$, $i=1,2$ denote harvest rates and natural
survival rates, respectively. The study in \cite{LP} shows that the system
(\ref{pc1})-(\ref{pc2}) under (\ref{C0}) generates a wide range of different
behaviors: the occurrence of periodic and chaotic behavior and phenomena such
as bubbles and the counter-intuitive \textquotedblleft hydra effect" (an
increase in harvesting yields an increase in the over-all population) are
established for the autonomous system%
\begin{align*}
x_{n+1}  &  =(1-h_{1})s_{1}y_{n}+(1-h_{2})s_{2}x_{n}\\
y_{n+1}  &  =(1-h_{1})bx_{n}e^{\alpha-(1-h_{1})\gamma x_{n}}.
\end{align*}

Our results in this paper complement the existing literature, e.g.
\cite{AJ}--\cite{GL}, \cite{LP} and \cite{Z}. In the next section we obtain
general results on the uniform boundedness and convergence to zero for the
non-autonomous system (\ref{pc1})-(\ref{pc2}). We also dicuss a refinement of
the convergence to zero results when the parameters of the system are periodic
(simulating extinction in a periodic environment). In particular, these
results show that convergence to zero occurs even if the mean value of
$\sigma_{2,n}$ exceeds 1.

In Section \ref{cr} we study the dynamics of orbits for a mathematically
interesting special case of (\ref{pc1})-(\ref{pc2}) in which $\sigma_{2,n}=0.$
This special case was studied with constant parameters (autonomous case) in
\cite{FHL} where conditions for the occurrence of a globally attracting
positive fixed point as well as a two-cycle (not globally attracting) were
obtained. Conditions implying the occurrence of the two-cycle are of
particular interest to us. In this case, the system reduces to a second-order
equation with a nonhyperbolic positive fixed point. A semiconjugate
factorization of this equation is known (see below) even with variable
parameters and we use it to prove the occurrence of complex dynamics,
including multiple stable (or multistable) periodic and non-periodic solutions
generated from different initial values. Our results also extend the
period-two result in \cite{FHL} to a wider parameter range while allowing some
parameters to be periodic.

\section{Uniform boundedness and global convergence to zero\label{exti}}

For the system (\ref{pc1})-(\ref{pc2}) we generally assume that for all
$n\geq0$:
\begin{align}
\alpha_{n},\beta_{n},\sigma_{i,n},c_{i,n}  &  \geq0,\quad i=1,2\label{rpars}\\
\beta_{n},\sigma_{1,n}  &  >0\ \text{for inifinitely many }n\nonumber
\end{align}

\subsection{General results}

We begin with a simple, yet useful lemma.

\begin{lemma}
\label{max}Let $\alpha>0$, $0<\beta<1$ and $x_{0}\geq0.$ If for all $n\geq0$
\begin{equation}
x_{n+1}\leq\alpha+\beta x_{n} \label{mx}%
\end{equation}
then for every $\varepsilon>0$ and all sufficiently large values of $n$
\[
x_{n}\leq\frac{\alpha}{1-\beta}+\varepsilon.
\]

\end{lemma}

\begin{proof}
Let $u_{0}=x_{0}$ and note that every solution of the linear, first-order
equation $u_{n+1}=\alpha+\beta u_{n}$ converges to its fixed point
$\alpha/(1-\beta).$ Further,%
\begin{align*}
x_{1}  &  \leq\alpha+\beta x_{0}=\alpha+\beta u_{0}=u_{1}\\
x_{2}  &  \leq\alpha+\beta x_{1}\leq\alpha+\beta u_{1}=u_{2}%
\end{align*}
and by induction, $x_{n}\leq u_{n}.$ Since $u_{n}\,\rightarrow\,\alpha
/(1-\beta)$ for every $\varepsilon>0$ and all sufficiently large $n$%
\[
x_{n}\leq u_{n}\leq\frac{\alpha}{1-\beta}+\varepsilon.
\]

\end{proof}

\medskip

The following result from the literature is quoted as a lemma. See \cite{LnAr}
for the proof and some background and references on this result which holds in
a more general setting than discussed here.

\begin{lemma}
\label{gen}Let $\alpha\in(0,1)$ and assume that the functions $f_{n}%
:[0,\infty)^{k+1}\rightarrow\lbrack0,\infty)$ satisfy the inequality%
\begin{equation}
f_{n}(u_{0},\ldots,u_{k})\leq\alpha\max\{u_{0},\ldots,u_{k}\} \label{fin}%
\end{equation}
for all $(u_{0},\ldots,u_{k})\in\lbrack0,\infty)$ and all $n\geq0$. Then for
every solution $\{x_{n}\}$ of the difference equation
\begin{equation}
x_{n+1}=f_{n}(x_{n},x_{n-1},\ldots,x_{n-k}) \label{geq}%
\end{equation}
the following is true%
\begin{equation}
x_{n}\leq\alpha^{n/(k+1)}\max\{x_{0},x_{-1}\ldots,x_{-k}\}. \label{zges}%
\end{equation}

\end{lemma}

Note that (\ref{fin}) implies that $x_{n}=0$ is a constant solution of
(\ref{geq}) and further, (\ref{zges}) implies that this solution is globally
exponentially stable.

\begin{theorem}
\label{ric}Assume that (\ref{rpars}) holds and further, let $\alpha_{n}$ be
bounded and $\limsup_{n\rightarrow\infty}\sigma_{2,n}<1$.

(a) If $\sigma_{1,n}$ is bounded and there is $M>0$ such that $\beta_{n}\leq
Mc_{1,n}$ for all $n\geq0$ then every orbit of (\ref{pc1})-(\ref{pc2}) in
$[0,\infty)^{2}$ is uniformly bounded.

(b) If $\beta_{n}$ is bounded and the following inequality holds then all
orbits of (\ref{pc1})-(\ref{pc2}) in $[0,\infty)^{2}$ converge to (0,0):%
\begin{equation}
\limsup_{n\rightarrow\infty}\left(  \sigma_{1,n}\beta_{n}e^{\alpha_{n}}%
+\sigma_{2,n}\right)  <1. \label{c0}%
\end{equation}

\end{theorem}

\begin{proof}
(a) For $u,v\geq0$ and all $n\geq0$ define%
\[
\phi_{n}(u,v)=\beta_{n}e^{\alpha_{n}-c_{1,n}u-c_{2,n}v}%
\]

If $c_{1,n}\not =0$ for all $n$ then elementary calculus yields%
\begin{equation}
u\phi_{n}(u,v)\leq u\phi_{n}\left(  \frac{1}{c_{1,n}},0\right)  =\frac
{\beta_{n}}{c_{1,n}}e^{\alpha_{n}-1} \label{phin}%
\end{equation}

If $c_{1,n}=0$ for some $n$ then $\beta_{n}\leq Mc_{1,n}=0$ and $\phi
_{n}(u,v)=0$ for such $n$.

Next, by the hypotheses there are numbers $M_{1},M_{2}>0$ and $\bar{\sigma}%
\in(0,1)$ such that for all sufficiently large values of $n$%
\[
\sigma_{1,n}\leq M_{1},\quad\alpha_{n}\leq M_{2},\quad\sigma_{2,n}\leq
\bar{\sigma}%
\]

Since $\beta_{n}\leq Mc_{1,n}$, it follows that for $u,v\geq0$ and all $n$
\[
u\phi_{n}(u,v)\leq Me^{M_{2}-1}\doteq M_{0}%
\]

It follows that $y_{n}\leq M_{0}$ for $n\geq1$ so by (\ref{pc1})%
\[
x_{n+1}\leq M_{0}M_{1}+\sigma_{2,n}(u,v)x_{n}\leq M_{0}M_{1}+\bar{\sigma}x_{n}%
\]

Next, applying Lemma \ref{max} with $\varepsilon=\bar{\sigma}/(1-\bar{\sigma
})$ we obtain for all (large) $n$
\[
0\leq x_{n}\leq\frac{M_{0}M_{1}+\bar{\sigma}}{1-\bar{\sigma}}%
\]
as claimed.

(b) If $\phi_{n}$ is as defined in (a) above then (\ref{pc2}) implies that%
\[
y_{n}\leq\beta_{n}e^{\alpha_{n}}x_{n-1}%
\]

By (\ref{c0}) there is $\delta\in(0,1)$ such that $\sigma_{1,n}\beta
_{n}e^{\alpha_{n}}+\sigma_{2,n}\leq\delta$ for all (large) $n$ so from
(\ref{pc1}) it follows that
\begin{align*}
x_{n+1}  &  \leq\beta_{n}e^{\alpha_{n}}\sigma_{1,n}x_{n-1}+\sigma_{2,n}x_{n}\\
&  \leq\left(  \sigma_{1,n}\beta_{n}e^{\alpha_{n}}+\sigma_{2,n}\right)
\max\{x_{n},x_{n-1}\}\\
&  \leq\delta\max\{x_{n},x_{n-1}\}
\end{align*}

Lemma \ref{gen} now implies that $\lim_{n\rightarrow\infty}x_{n}=0.$ Further,
since both $\alpha_{n}$ and $\beta_{n}$ are bounded, there is $\mu>0$ such
that $\beta_{n}e^{\alpha_{n}}\leq\mu$ for all $n$. Thus,
\[
\lim_{n\rightarrow\infty}y_{n}\leq\mu\lim_{n\rightarrow\infty}x_{n-1}=0
\]
and the proof is complete.
\end{proof}

\medskip

\begin{remark}
1. The hypotheses of the above theorem allow the parameters to contain
arbitrary low-level fluctuations, a feature of possible interest in some
modeling applications.

2. In Part (a) of the above corollary it is more essential to have
$c_{1,n}\not =0$ than $\beta_{n}$ be bounded. Indeed, unbounded solutions
occur in the following autonomous linear system%
\begin{align*}
x_{n+1}  &  =\sigma_{1}y_{n}+\sigma_{2}x_{n}\\
y_{n+1}  &  =\beta e^{\alpha}x_{n}%
\end{align*}
in which $c_{1,n}=0$ for all $n$ and $\beta_{n}=\beta$ is bounded. Note that%
\[
x_{n+2}=\sigma_{1}y_{n+1}+\sigma_{2}x_{n+1}=\beta e^{\alpha}\sigma_{1}%
x_{n}+\sigma_{2}x_{n+1}%
\]

It is evident that unbounded solutions exist unless $\sigma_{1}\beta
e^{\alpha}\leq1-\sigma_{2}$. This is a severe restriction resembling that in
Part (b) of the above corollary.
\end{remark}

\subsection{Global convergence to zero with periodic parameters}

Theorem \ref{ric} gives general sufficient conditions for the convergence of
all non-negative orbits of the planar system to (0,0). In this section we
assume that all parameters are periodic and study convergence to zero in this
more restricted setting. In particular, the results in this section indicate
that global convergence to zero may occur even if (\ref{c0}) does not hold;
see Section \ref{bio} below. Recall from the proof of Theorem \ref{ric} that%
\begin{equation}
x_{n+1}\leq\beta_{n}e^{\alpha_{n}}\sigma_{1,n}x_{n-1}+\sigma_{2,n}%
x_{n}.\label{n2l}%
\end{equation}

The right hand side of the above inequality is a linear expression. Consider
the linear difference equation%
\begin{equation}
u_{n+1}=a_{n}u_{n}+b_{n}u_{n-1},\qquad a_{n+p_{1}}=a_{n},\ b_{n+p_{2}}=b_{n}
\label{linp}%
\end{equation}
where the sequences $a_{n},b_{n}$ have periods $p_{1},p_{2}$ that are positive
integers. If $p=\operatorname{lcm}(p_{1},p_{2})$ is the least common multiple
of the two periods, we say that the linear difference equation (\ref{linp}) is
periodic with period $p.$ We assume that%
\begin{equation}
a_{n},b_{n}\geq0,\quad n=0,1,2,\ldots\label{pos}%
\end{equation}

In the biological setting, these parameters are defined as follows:
\begin{equation}
a_{n}=\sigma_{2,n},\quad b_{n}=\beta_{n}e^{\alpha_{n}}\sigma_{1,n}%
\label{bioab}%
\end{equation}

Of interest is the fact that the biological parameters $\alpha_{n},\beta
_{n},\sigma_{1,n}$ need not be periodic in order for $a_{n},b_{n}$ to be
periodic. As long as the combination of parameters $\beta_{n}e^{\alpha_{n}%
}\sigma_{1,n}$ is periodic along with $\sigma_{2,n}$ we obtain periodicity.
This allows greater flexibility in defining some of the system parameters.

By Lemma \ref{gen} every solution of (\ref{linp}) converges to zero if
$a_{n}+b_{n}<1$ for all $n.$ However, it is known that convergence to zero may
occur even when $a_{n}+b_{n}$ exceeds 1 (for infinitely many $n$ in the
periodic case). We use the approach in \cite{Lin} to examine the consequences
of this issue when the planar system has periodic parameters. The following
result is an immediate consequence of Theorem 13 in \cite{Lin}.

\begin{lemma}
\label{scl}Assume that (\ref{linp}) has period $p\geq1$ and $\delta_{j}%
,\theta_{j}$ for $j=1,2,\ldots,p$ are obtained by iteration from the real
initial values%
\begin{equation}
\delta_{0}=0,\ \delta_{1}=1;\quad\theta_{0}=1,\ \theta_{1}=0 \label{01}%
\end{equation}

Suppose that the quadratic polynomial%
\begin{equation}
\delta_{p}r^{2}+(\theta_{p}-\delta_{p+1})r-\theta_{p+1}=0 \label{qchr}%
\end{equation}
is proper, i.e. not $0=0$ and has a real root $r_{1}\not =0.$ If the
recurrence%
\begin{equation}
r_{n+1}=a_{n}+\frac{b_{n}}{r_{n}} \label{epr}%
\end{equation}
\noindent generates nonzero real numbers $r_{2},\ldots,r_{p}$ then
$\{r_{n}\}_{n=1}^{\infty}$ is periodic with preiod $p$ and yields a
semiconjuagte factorization of (\ref{linp}) into a pair of first order
equations as follows:
\begin{align}
t_{n+1}  &  =-\frac{b_{n}}{r_{n}}t_{n},\quad t_{1}=u_{1}-r_{1}u_{0}%
\label{teq}\\
u_{n+1}  &  =r_{n+1}u_{n}+t_{n+1}. \label{xeq}%
\end{align}

\end{lemma}

For an introduction to the concept of semiconjuagte factorization see
\cite{FSOR} which also contains the application of this method to linear
equations over algebraic fields. A more general application of semiconjugate
factorization to linear equations in rings appeares in \cite{Lin}.

The sequence $\{r_{n}\}$ that is generated by (\ref{epr}) is said to be
an\textit{ eigensequence} of (\ref{linp}). Eigenvalues are constant
eigensequences, since if $p=1$ in Lemma \ref{scl} then (\ref{qchr}) reduces to%
\[
r^{2}-\delta_{2}r-\theta_{2}=0\quad\text{or\quad}r^{2}-a_{1}r-b_{1}=0
\]

The last equation is recognizable as the charateristic polynomial of
(\ref{linp}).

Each of the equations (\ref{teq}) and (\ref{xeq}) readily yields a solution by
iteration as follows%
\begin{align}
t_{n}  &  =t_{1}(-1)^{n-1}\left(  \frac{b_{1}b_{2}\cdots b_{n-1}}{r_{1}%
r_{2}\cdots r_{n-1}}\right)  ,\label{tn}\\
u_{n}  &  =r_{n}r_{n-1}\cdots r_{2}u_{1}+r_{n}r_{n-1}\cdots r_{3}t_{2}+\cdots
r_{n}t_{n-1}+t_{n}\nonumber\\
&  =r_{n}r_{n-1}\cdots r_{2}r_{1}u_{0}+\sum_{i=1}^{n-1}r_{n}r_{n-1}\cdots
r_{i+1}t_{i}+t_{n} \label{un}%
\end{align}

\begin{lemma}
\label{b0}Suppose that the numbers $\delta_{n}$ and $\theta_{n}$ are defined
as in Lemma \ref{scl}, although here we do not assume that (\ref{linp}) is
periodic. Then

(a) $\theta_{n}=0$ for all $n\geq2$ if and only if $b_{1}=0.$

(b) If (\ref{pos})\ holds then for all $n\geq2$%
\begin{align}
\delta_{n}  &  \geq a_{1}a_{2}\cdots a_{n-1},\quad\theta_{n}\geq b_{1}%
a_{2}\cdots a_{n-1}\label{anb}\\
\delta_{2n-1}  &  \geq b_{2}b_{4}\cdots b_{2n-2},\quad\theta_{2n}\geq
b_{1}b_{3}\cdots b_{2n-1} \label{anc}%
\end{align}

\end{lemma}

\begin{proof}
(a) Let $b_{1}=0.$ Then $\theta_{2}=b_{1}=0$ and since $\theta_{1}=0$ by
definition it follows that $\theta_{3}=0.$ Induction completes the proof that
$\theta_{n}=0$ if $n\geq2.$ The converse is obvious since $b_{1}=\theta_{2}.$

(b) Since $\delta_{2}=a_{1}$ and $\theta_{2}=b_{1}$ the stated inequalities
hold for $n=2.$ If (\ref{anb}) is true for some $k\geq2$ then%
\begin{align*}
\delta_{k+1}  &  =a_{k}\delta_{k}+b_{k}\delta_{k-1}\geq a_{k}\delta_{k}\geq
a_{1}a_{2}\cdots a_{k-1}a_{k}\\
\theta_{k+1}  &  =a_{k}\theta_{k}+b_{k}\theta_{k-1}\geq a_{k}\theta_{k}\geq
b_{1}a_{2}\cdots a_{k-1}a_{k}%
\end{align*}

Now, the proof is completed by induction. The proof of (\ref{anc}) is similar
since%
\[
\delta_{3}=a_{2}\delta_{2}+b_{2}\delta_{1}\geq b_{2}\quad\text{and\quad}%
\theta_{4}=a_{3}\theta_{3}+b_{3}\theta_{2}\geq b_{3}b_{1}%
\]
and if (\ref{anc}) holds for some $k\geq2$ then%
\begin{align*}
\delta_{2k+1}  &  \geq b_{2k}\delta_{2k-1}\geq b_{2}b_{4}\cdots b_{2k-2}%
b_{2k}\\
\theta_{2k+2}  &  \geq b_{2k+1}\theta_{2k}\geq b_{1}b_{3}\cdots b_{2k-1}%
b_{2k+1}%
\end{align*}
which establishes the induction step.
\end{proof}

\medskip

\begin{lemma}
\label{lin}Assume that (\ref{pos})\ holds with $a_{i}>0$ for $i=1,\ldots,p$
and (\ref{linp}) is periodic with period $p\geq2.$ Then

(a) Equation (\ref{linp}) has a positive eigensequence $\{r_{n}\}$ of period
$p.$

(b) If $b_{i}>0$ for $i=1,\ldots,p$ then%
\begin{equation}
r_{1}r_{2}\cdots r_{p}=\frac{1}{2}\left(  \delta_{p+1}+\theta_{p}%
+\sqrt{(\delta_{p+1}-\theta_{p})^{2}+4\delta_{p}\theta_{p+1}}\right)
\label{ala}%
\end{equation}

Hence, $r_{1}r_{2}\cdots r_{p}<1$ if
\begin{equation}
\delta_{p}\theta_{p+1}<(1-\delta_{p+1})(1-\theta_{p}) \label{alb}%
\end{equation}

(c) If $b_{i}<1$ for $i=1,\ldots,p$ then $r_{1}r_{2}\cdots r_{p}>b_{1}%
b_{2}\cdots b_{p}.$
\end{lemma}

\begin{proof}
(a) Lemma \ref{b0} shows that $\delta_{i}>0$ for $i=2,\ldots,p+1.$ Now, either
(i) $b_{1}>0$ or (ii) $b_{1}=0.$ In case (i), the root $r^{+}$ of the
quadratic polynomial (\ref{qchr}) is positive since by Lemma \ref{b0}
$\theta_{p+1}>0$ and thus%
\[
r^{+}=\frac{\delta_{p+1}-\theta_{p}+\sqrt{(\delta_{p+1}-\theta_{p}%
)^{2}+4\delta_{p}\theta_{p+1}}}{2\delta_{p}}>\frac{\delta_{p+1}-\theta
_{p}+\left\vert \delta_{p+1}-\theta_{p}\right\vert }{2\delta_{p}}\geq0.
\]

If $r_{1}=r^{+}$ then from (\ref{epr}) $r_{i}=a_{i-1}+b_{i-1}/r_{i-1}\geq
a_{i-1}>0$ for $i=2,\ldots,p+1.$ Thus by Lemma \ref{scl}, (\ref{linp}) has a
unitary (in fact, positive) eigensequence of period $p.$ If $b_{1}=0$ then by
Lemma \ref{b0} $\theta_{p}=\theta_{p+1}=0$ and (\ref{qchr}) reduces to%
\[
\delta_{p}r^{2}-\delta_{p+1}r=0
\]
which has a root $r^{+}=\delta_{p+1}/\delta_{p}>0.$ As in the previous case it
follows that (\ref{linp}) has a positive eigensequence of period $p.$

(b) To estalish (\ref{ala}), let $r_{1}=r^{+}$ and note that (\ref{qchr}) can
be written as%
\begin{equation}
r_{1}=\frac{\delta_{p+1}r_{1}+\theta_{p+1}}{\delta_{p}r_{1}+\theta_{p}}
\label{r1}%
\end{equation}

Since $\{r_{n}\}$ has period $p$, $r_{p+1}=r_{1}$ so from (\ref{epr}) and the
definition of the numbers $\delta_{n}$ and $\theta_{n}$ it follows that%
\begin{align*}
a_{p}+\frac{b_{p}}{r_{p}}  &  =r_{p+1}=\frac{\delta_{p+1}r_{1}+\theta_{p+1}%
}{\delta_{p}r_{1}+\theta_{p}}=\frac{(a_{p}\delta_{p}+b_{p}\delta_{p-1}%
)r_{1}+a_{p}\theta_{p}+b_{p}\theta_{p-1}}{\delta_{p}r_{1}+\theta_{p}}\\
&  =\frac{a_{p}(\delta_{p}r_{1}+\theta_{p})+b_{p}(\delta_{p-1}r_{1}%
+\theta_{p-1})}{\delta_{p}r_{1}+\theta_{p}}\\
&  =a_{p}+\frac{b_{p}}{(\delta_{p}r_{1}+\theta_{p})/(\delta_{p-1}r_{1}%
+\theta_{p-1})}%
\end{align*}

Since $b_{p}\not =0$ it follows that%
\[
r_{p}=\frac{\delta_{p}r_{1}+\theta_{p}}{\delta_{p-1}r_{1}+\theta_{p-1}}%
\]

We claim that if $b_{i}\not =0$ for $i=1,\ldots,p$ then%
\begin{equation}
r_{p-j}=\frac{\delta_{p-j}r_{1}+\theta_{p-j}}{\delta_{p-j-1}r_{1}%
+\theta_{p-j-1}},\quad j=0,1,\ldots,p-2 \label{rj}%
\end{equation}

This claim is easily seen to be true by induction; we showed that it is true
for $j=0$ and if (\ref{rj}) holds for some $j$ then by (\ref{epr})
\begin{align*}
a_{p-j-1}+\frac{b_{p-j-1}}{r_{p-j-1}}  &  =r_{p-j}=\frac{(a_{p-j-1}%
\delta_{p-j-1}+b_{p-j-1}\delta_{p-j-2})r_{1}+(a_{p-j-1}\theta_{p-j-1}%
+b_{p-j-1}\theta_{p-j-2})}{\delta_{p-j-1}r_{1}+\theta_{p-j-1}}\\
&  =\frac{a_{p-j-1}(\delta_{p-j-1}r_{1}+\theta_{p-j-1})+b_{p-j-1}%
(\delta_{p-j-2}r_{1}+\theta_{p-j-2})}{\delta_{p-j-1}r_{1}+\theta_{p-j-1}}\\
&  =a_{p-j-1}+\frac{b_{p-j-1}(\delta_{p-j-2}r_{1}+\theta_{p-j-2})}%
{\delta_{p-j-1}r_{1}+\theta_{p-j-1}}%
\end{align*}
from which it follows that
\[
r_{p-j-1}=\frac{\delta_{p-j-1}r_{1}+\theta_{p-j-1}}{\delta_{p-j-2}r_{1}%
+\theta_{p-j-2}}%
\]
and the induction argument is complete. Now, using (\ref{rj}) we obtain%
\begin{equation}
r_{p}r_{p-1}\cdots r_{2}r_{1}=\frac{\delta_{p}r_{1}+\theta_{p}}{\delta
_{p-1}r_{1}+\theta_{p-1}}\frac{\delta_{p-1}r_{1}+\theta_{p-1}}{\delta
_{p-2}r_{1}+\theta_{p-2}}\cdots\frac{\delta_{2}r_{1}+\theta_{2}}{\delta
_{1}r_{1}+\theta_{1}}r_{1}=\delta_{p}r_{1}+\theta_{p} \label{rs}%
\end{equation}

Given that $r_{1}=r^{+}$ (\ref{rs}) implies that%
\begin{align*}
r_{1}r_{2}\cdots r_{p}  &  =\delta_{p}\frac{\delta_{p+1}-\theta_{p}%
+\sqrt{(\delta_{p+1}-\theta_{p})^{2}+4\delta_{p}\theta_{p+1}}}{2\delta_{p}%
}+\theta_{p}\\
&  =\frac{1}{2}\left(  \delta_{p+1}+\theta_{p}+\sqrt{(\delta_{p+1}-\theta
_{p})^{2}+4\delta_{p}\theta_{p+1}}\right)
\end{align*}
and (\ref{ala}) is obtained. Hence, $r_{1}r_{2}\cdots r_{p}<1$ if%
\[
\delta_{p+1}+\theta_{p}+\sqrt{(\delta_{p+1}-\theta_{p})^{2}+4\delta_{p}%
\theta_{p+1}}<2
\]

Upon rearranging terms and squaring:%
\[
(\delta_{p+1}-\theta_{p})^{2}+4\delta_{p}\theta_{p+1}<4-4(\delta_{p+1}%
+\theta_{p})+(\delta_{p+1}+\theta_{p})^{2}%
\]
which reduces to (\ref{alb}) after straightforward algebraic manipulations.

(c) First, assume that $p$ is odd. Then by (\ref{anc})%
\[
\delta_{p}\theta_{p+1}=(b_{2}b_{4}\cdots b_{p-1})(b_{1}b_{3}\cdots
b_{p})=b_{1}b_{2}\cdots b_{p}%
\]
so from (\ref{ala})%
\[
r_{1}r_{2}\cdots r_{p}>\sqrt{\delta_{p}\theta_{p+1}}=\sqrt{b_{1}b_{2}\cdots
b_{p}}%
\]

If $b_{i}<1$ for $i=1,\ldots,p$ then $b_{1}b_{2}\cdots b_{p}<1$ so
$\sqrt{b_{1}b_{2}\cdots b_{p}}>b_{1}b_{2}\cdots b_{p}$ as required. Now let
$p$ be even. Then from (\ref{ala}) and (\ref{anc})
\[
r_{1}r_{2}\cdots r_{p}>\frac{\delta_{p+1}+\theta_{p}}{2}\geq\frac{b_{2}%
b_{4}\cdots b_{p}+b_{1}b_{3}\cdots b_{p-1}}{2}%
\]

If $b_{i}<1$ for $i=1,\ldots,p$ then $b_{2}b_{4}\cdots b_{p}\geq b_{1}%
b_{2}\cdots b_{p}$ and$\ b_{1}b_{3}\cdots b_{p-1}\geq b_{1}b_{2}\cdots b_{p}$
and the proof is complete.
\end{proof}

\medskip

\begin{theorem}
\label{bext}Assume that the sequences $\beta_{n}e^{\alpha_{n}}\sigma_{1,n}$
and $\sigma_{2,n}$ are strictly positive and periodic and let $p$ be the least
common multiple of their periods. All non-negative orbits of (\ref{pc1}%
)-(\ref{pc2}) converge to (0,0) if $\beta_{i}e^{\alpha_{i}}\sigma_{1,i}<1$ for
$i=1,\ldots,p$ and (\ref{alb}) holds.
\end{theorem}

\begin{proof}
Let $\{u_{n}\}$ be a solution of the linear equation (\ref{linp}) with
$a_{n},b_{n}$ defined by (\ref{bioab}). If $u_{0}=x_{0}$ and $u_{1}=x_{1}$
then by (\ref{n2l})%
\begin{align*}
x_{2} &  \leq\beta_{0}e^{\alpha_{0}}\sigma_{1,1}x_{0}+\sigma_{2,1}x_{1}%
=\beta_{0}e^{\alpha_{0}}\sigma_{1,1}u_{0}+\sigma_{2,1}u_{1}=u_{2}\\
x_{3} &  \leq\beta_{1}e^{\alpha_{1}}\sigma_{1,2}x_{2}+\sigma_{2,2}x_{2}%
\leq\beta_{1}e^{\alpha_{1}}\sigma_{1,2}u_{1}+\sigma_{2,2}u_{2}=u_{3}%
\end{align*}

By induction it follows that $x_{n}\leq u_{n}.$ If (\ref{alb}) holds then by
Lemma \ref{lin}, $\lim_{n\rightarrow\infty}u_{n}=0$ so $\{x_{n}\}$ converges
to 0. Further, $\lim_{n\rightarrow\infty}y_{n}=0$ as in the proof of Theorem
\ref{ric} and the proof is complete.
\end{proof}

\medskip

Recall that the individual sequences $\alpha_{n},\beta_{n},\sigma_{1,n}$ need
not be periodic; see the note following (\ref{bioab}). Therefore,
\textit{Theorem \ref{bext} applies to the system (\ref{pc1})-(\ref{pc2}) even
if the system itself is not periodic} as long as the combination $\beta
_{n}e^{\alpha_{n}}\sigma_{1,n}$ of parameters is periodic along with
$\sigma_{2,n}.$

\subsection{Stocking strategies that do not prevent extinction\label{bio}}

Condition (\ref{alb}) involves the numbers $\delta_{j},\theta_{j}$ rather than
the coefficients of (\ref{linp}) directly. To illustrate the biological
significance of this condition with regard to extinction, consider the case of
period $p=2$ in which the role of $a_{i}$, $b_{i}$ is more apparent.
Inequality (\ref{alb}) in this case is%
\begin{align*}
\delta_{2}\theta_{3} &  <(1-\delta_{3})(1-\theta_{2})\\
a_{1}a_{2}b_{1} &  <(1-b_{2}-a_{1}a_{2})(1-b_{1})
\end{align*}

Simple manipulations reduce the last inequality to%
\begin{equation}
a_{1}a_{2}<(1-b_{1})(1-b_{2}).\label{p2}%
\end{equation}

In this form, it is easy to see the signficance of (\ref{alb}) with regard to
extinction. For if $b_{1},b_{2}<1$ then (\ref{p2}) holds even if $a_{1}>1$ or
$a_{2}>1$ so global convergence to (0,0) my occur when (\ref{c0}) does not
hold. Further, it is possible that (\ref{p2}) holds, together with arbitrarily
large mean value
\begin{equation}
\frac{a_{1}+a_{2}}{2}>1\label{pav}%
\end{equation}
if, say $a_{1}\rightarrow0$ as $a_{2}\rightarrow\infty$. In population models
this implies that if (\ref{p2}) holds with
\[
a_{i}=\sigma_{2,i},\quad b_{i}=\beta_{i}e^{\alpha_{i}}\sigma_{1,i}\qquad i=1,2
\]
then extinction may still occur after restocking the adult population so that
the mean value of the composite parameter $\sigma_{2,n}$ exceeds unity by a
wide margin.

\section{Complex multistable behavior\label{cr}}

In this section we consider the reduced system
\begin{align}
x_{n+1}  &  =\sigma_{1,n}y_{n}\label{rm1}\\
y_{n+1}  &  =\beta_{n}x_{n}e^{\alpha_{n}-c_{1,n}x_{n}-c_{2,n}y_{n}}
\label{rm2}%
\end{align}
where we assume that
\begin{equation}
\sigma_{1,n},c_{1,n},c_{2,n},\beta_{n}>0,\text{\quad}\alpha_{n}\geq0.
\label{rmp}%
\end{equation}

In the context of stage-structured models\ the assumption $\sigma_{2,n}=0$
applies in particular, to the case of a semelparous species, i.e. an organism
that reproduces only once before death. Additional interpretations in terms of
harvesting, migrations or other factors may be possible if $\sigma_{2,n}$
includes additional factors beyond the natural adult survival rate.

The system (\ref{rm1})-(\ref{rm2}) with $c_{2,n}=0$ has been studied in the
literature; for instance, an autonomous version is discussed in \cite{LP} and
\cite{Z}. The assumption $c_{2,n}>0,$ which adds greater inter-species
competition into the stage-structured model, leads to theoretical issues that
are not well-understood. We proceed by folding he system (\ref{rm1}%
)-(\ref{rm2}) to a second-order difference equation. The process here is
simple and self-contained but for a broader introduction and other
applications of folding to the study of discrete planar systems we refer to
\cite{DsPl}.

From (\ref{rm1}) we obtain $y_{n}=x_{n+1}/\sigma_{1,n}$. Now using (\ref{rm1})
and (\ref{rm2}) we obtain:%
\[
x_{n+2}=\sigma_{1,n+1}\beta_{n}x_{n}e^{\alpha_{n}-c_{1,n}x_{n}-c_{2,n}y_{n}%
}=\sigma_{1,n}\beta_{n}x_{n}e^{\alpha_{n}-c_{1,n}x_{n}-(c_{2,n}/\sigma
_{1,n})x_{n+1}}%
\]

This can be written more succinctly as%
\begin{equation}
x_{n+1}=x_{n-1}e^{a_{n}-c_{1,n}x_{n-1}-(c_{2,n}/\sigma_{1,n})x_{n}}
\label{o2s}%
\end{equation}

where%
\[
a_{n}=\alpha_{n}+\ln(\beta_{n}\sigma_{1,n+1}).
\]

\subsection{Fixed points, global stability}

It is useful to start by examining the fixed points of (\ref{o2s}) when all
parameters are constants, i.e. if (\ref{rm1})-(\ref{rm2}) is an autonomous
system. Then (\ref{o2s}) takes the form of the autonomous difference equation:%
\begin{equation}
x_{n+1}=x_{n-1}e^{a-c_{1}x_{n-1}-(c_{2}/\sigma_{1})x_{n}} \label{rmf}%
\end{equation}

This equation clearly has a fixed point at 0. The following is consequence of
Theorem \ref{ric}(b).

\begin{corollary}
\label{0gas}Assume that the system (\ref{rm1})-(\ref{rm2}) is autonomous, i.e.
$\alpha_{n}=\alpha$, $\beta_{n}=\beta$, $\sigma_{1,n}=\sigma_{1}$,
$c_{1,n}=c_{1}$ and $c_{2,n}=c_{2}$ are constants for all $n$.

(a) If $a=\alpha+\ln(\beta\sigma_{1})<0$ then 0 is the unique fixed point of
(\ref{rmf}) in $[0,\infty)$ and all positive solutions of (\ref{rmf}) converge
to zero$.$

(b) The eigenvalues of the linearization of (\ref{rmf}) at 0 are $\pm e^{a/2}%
$; thus, 0 is locally asymptotically stable if $a<0.$
\end{corollary}

If $a>0$ then (\ref{rmf}) has exactly two fixed points: 0 and a positive fixed
point
\[
\bar{x}=\frac{a\sigma_{1}}{c_{1}\sigma_{1}+c_{2}}.
\]

Substituting $r_{n}=c_{1}x_{n}$ in (\ref{rmf}) yields%
\begin{equation}
r_{n+1}=r_{n-1}e^{a-r_{n-1}-br_{n}},\quad b=\frac{c_{2}}{\sigma_{1}c_{1}}
\label{fhl}%
\end{equation}

The positive fixed point of this equation is
\[
\bar{r}=\frac{a}{1+b}=c_{1}\bar{x}.
\]

The next result is proved in \cite{FHL}.

\begin{theorem}
\label{gas}Let $a\in(0,1]$.

(a) If $b\in(0,1)$ (i.e. $c_{2}<\sigma_{1}c_{1}$) then the positive fixed
point $\bar{r}$ of (\ref{fhl}) is a global attractor of all of its positive solutions.

(b) If $b=1$ (i.e. $c_{2}=\sigma_{1}c_{1}$) then every non-constant, positive
solution of (\ref{fhl}) converges to a 2-cycle whose consecutive points
satisfy $r_{n}+r_{n+1}=a,$ i.e. the mean value of the limit cycle is the fixed
point $\bar{r}=a/2.$
\end{theorem}

The two-cycle in Theorem \ref{gas}(b) is not unique--it is determined by the
initial values. We derive the precise mechanism that explains this, and much
more complex behavior below. In particular, we extend Part (b) of Theorem
\ref{gas} by showing that it holds for $a\in(0,2]$ and even some parameters
need not be constants.

\subsection{Order reduction}

The semiconjugate factorization method that we used earlier for linear
equations also applies to (\ref{o2s}) if the following condition holds:
\begin{equation}
c_{2,n}=\sigma_{1,n}c_{1,n}\quad n=0,1,2,\ldots\label{ms}%
\end{equation}

In the autonomous case this reduces to the condition in Theorem \ref{gas}(b),
i.e. $c_{2}=\sigma_{1}c_{1}.$ This condition that is restrictive but
admissible in a biological sense, leads to interesting nonhypberbolic dynamics
that we explore in the remainder of this paper. 

If (\ref{ms}) holds then we substitute $r_{n}=c_{1,n}x_{n}$ in (\ref{o2s}) to
obtain
\[
r_{n+1}=\frac{c_{1,n+1}}{c_{1,n-1}}r_{n-1}e^{a_{n}-r_{n-1}-r_{n}}%
\]
which can be written as
\begin{align}
r_{n+1} &  =r_{n-1}e^{d_{n}-r_{n-1}-r_{n}}\label{rms}\\
d_{n} &  =a_{n}+\ln[c_{1,n+1}/c_{1,n-1}].\nonumber
\end{align}

Note that if $c_{1,n}$ has period 2 or is constant then $c_{1,n+1}=c_{1,n-1}$
so $d_{n}=a_{n}.$ In any case, a solution $x_{n}=r_{n}/c_{1,n}$ of (\ref{o2s})
is derived in terms of a solution of (\ref{rms}) when (\ref{ms}) holds.

Equation (\ref{rms}) admits a semiconjugate factorization that splits it into
two equations of order one. Using the concept of form symmetry from
\cite{FSOR}, we define%
\[
t_{n}=\frac{r_{n}}{r_{n-1}e^{-r_{n-1}}}%
\]
for each $n\geq1$ and note that
\[
t_{n+1}t_{n}=\frac{r_{n+1}}{r_{n}e^{-r_{n}}}\frac{r_{n}}{r_{n-1}e^{-r_{n-1}}%
}=\frac{r_{n+1}}{r_{n-1}e^{-r_{n-1}-r_{n}}}=e^{d_{n}}%
\]
or equivalently,%
\begin{equation}
t_{n+1}=\frac{e^{d_{n}}}{t_{n}}. \label{sc1}%
\end{equation}

Now%
\begin{equation}
r_{n+1}=e^{d_{n}}r_{n-1}e^{-r_{n-1}}e^{-r_{n}}=e^{d_{n}}\frac{r_{n}}{t_{n}%
}e^{-r_{n}}=\frac{e^{d_{n}}}{t_{n}}r_{n}e^{-r_{n}}=t_{n+1}r_{n}e^{-r_{n}}
\label{sc2}%
\end{equation}

The pair of equations (\ref{sc1}) and (\ref{sc2}) constitute the semiconjugate
factorization of (\ref{rms}):%
\begin{align}
t_{n+1}  &  =\frac{e^{d_{n}}}{t_{n}},\quad t_{0}=\frac{r_{0}}{r_{-1}%
e^{-r_{-1}}}\label{star1}\\
r_{n+1}  &  =t_{n+1}r_{n}e^{-r_{n}} \label{star2}%
\end{align}

Every solution $\{r_{n}\}$ of (\ref{rms}) is generated by a solution of the
system (\ref{star1})-(\ref{star2}). Using the initial values $r_{-1},r_{0}$ we
obtain a solution $\{t_{n}\}$ of the first-order equation (\ref{star1}). This
solution is then used to obtain a solution of (\ref{star2}), and thus also of
(\ref{rms}).

\subsection{Complex behavior for the autonomous equation}

If $p=1$ then $d_{n}$ is constant, say $d_{n}=d$ for all $n$. In this case
(\ref{rms}) reduces to the autonomous equation:
\begin{equation}
r_{n+1}=r_{n-1}e^{d-r_{n-1}-r_{n}} \label{rmsa}%
\end{equation}
although (\ref{o2s}) may not be autonomous, e.g. if $c_{1,n}$ has period 2, as
noted above.\ 

If $d<0$ then Corollary \ref{0gas} implies that all solutions of (\ref{rmsa})
converge to 0. Let $d>0$ so that there is a positive fixed point
\[
\bar{r}=\frac{d}{2}>0.
\]

The eigenvalues of the linearization of (\ref{rmsa}) at $\bar{r}$ are $-1$ and
$-d/2$, showing in particular that $\bar{r}$ is nonhyperbolic. The behavior of
solutions of (\ref{rmsa}) is sufficiently unusual that we use the numerical
simulation depicted in Figure \ref{one} to motivate the subsequent discussion.

\begin{figure}[ptbh]
\centering
\includegraphics[width=6.03in,height=2.76in,keepaspectratio]{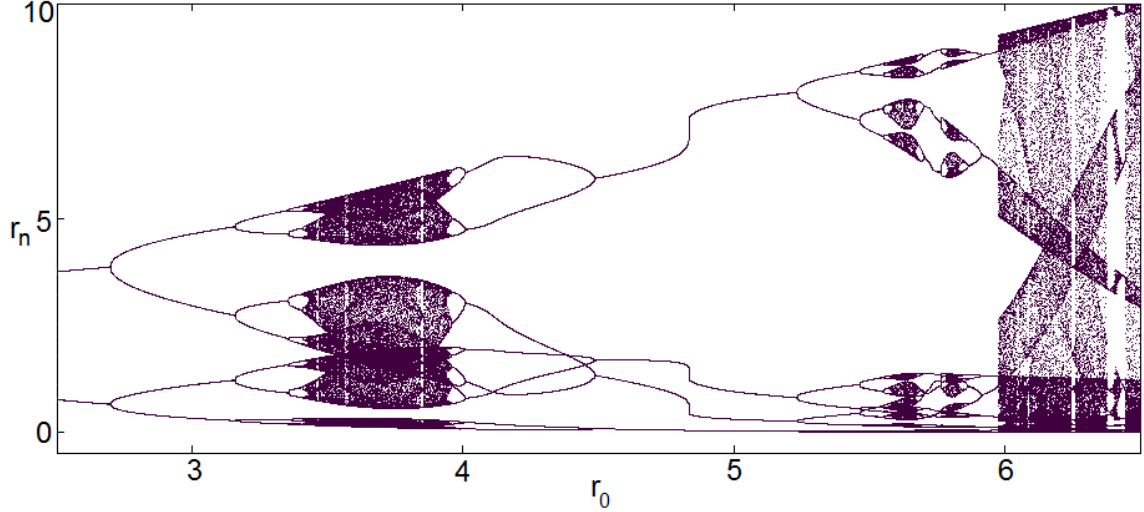}
\caption{Bifurcation of multiple stable solutions in the state-space}%
\label{one}%
\end{figure}

In Figure \ref{one}, $d=4.5$, $r_{-1}=d/2=2.25$ is fixed and $r_{0}%
\in(0,\infty)$ acts as a bifurcation parameter. The changing values of $r_{0}$
are shown on the horizontal axis in the range 2.5 to 6.5. For every grid value
of $r_{0}$ in the indicated range, 300 points of the corresponding solution
$\{r_{n}\}$ are plotted vertically. In this figure, coexisting solutions with
periods 2, 4, 8 and 16 are easily identified. The solutions shown in Figure
\ref{one} are stable since they are generated by numerical simulation, so that
qualitatively different, stable solutions exist for (\ref{rmsa}) \textit{for
different initial values}. In the remainder of this section we explain this
abundance of multistable solutions for (\ref{rmsa}) using the reduction
(\ref{star1})-(\ref{star2}).

All solutions of (\ref{star1}) with constant $d_{n}=d$ and $t_{0}\neq e^{d/2}$
are periodic with period 2:%
\[
\left\{  t_{0},\frac{e^{d}}{t_{0}}\right\}  =\left\{  \frac{r_{0}}%
{r_{-1}e^{-r_{-1}}},\frac{r_{-1}e^{d-r_{-1}}}{r_{0}}\right\}  .
\]

Hence the orbit of each nontrivial solution $\{r_{n}\}$ of (\ref{rmsa}) in its
state-space, namely, the $(r_{n},r_{n+1})$-plane, is restricted to the class
of curve-pairs%
\begin{equation}
g_{0}(r,t_{0})=t_{0}re^{-r}\quad\text{and\quad}g_{1}(r,t_{0})=t_{1}%
re^{-r},\quad t_{1}=\frac{e^{d}}{t_{0}} \label{ic}%
\end{equation}

These one-dimensional mappings form the building blocks of the
two-dimensional, standard state-space map $F$ of (\ref{rmsa}), i.e.
\[
F(u,r)=(r,ue^{d-u-r}).
\]

There are, of course, an infinite number of initial value-dependent
curve-pairs for the map $F.$

The next result indicates the specific mechanism for generating the solutions
of (\ref{rmsa}) from its semiconjugate factorization.

\begin{lemma}
\label{spt}Let $d>0$ and let $\{r_{n}\}$ be a solution of (\ref{rmsa}) with
initial values $r_{-1},r_{0}>0$.

(a) For $k=0,1,2,\ldots$ and $t_{0}$ as defined in (\ref{star1})
\[
r_{2k+1}=g_{1}\circ g_{0}(r_{2k-1},t_{0}),\quad r_{2k+2}=g_{0}\circ
g_{1}(r_{2k},t_{0})
\]

Thus, the odd terms of every solution of (\ref{rmsa}) are generated by the
class of one-dimensional maps $g_{1}\circ g_{0}$ and the even terms by
$g_{0}\circ g_{1}$;

(b) If the initial values \thinspace$r_{-1},r_{0}$ satisfy
\begin{equation}
r_{0}=r_{-1}e^{d/2-r_{-1}} \label{21}%
\end{equation}
then $g_{0}(r,t_{0})=g_{1}(r,t_{0})=re^{d/2-r}$; i.e. the two curves $g_{0}$
and $g_{1}$ coincide with the curve
\[
g(r)\doteq re^{d/2-r}%
\]

The trace of $g$ contains the fixed point $(\bar{r},\bar{r})$ in the
state-space and is invariant under $F.$
\end{lemma}

\begin{proof}
(a) For $k=0,1,2,\ldots$ (\ref{star2}) implies that%
\begin{align*}
r_{2k+1}  &  =t_{2k+1}r_{2k}e^{-r_{2k}}=t_{1}r_{2k}e^{-r_{2k}}=g_{1}%
(r_{2k},t_{0})\\
r_{2k}  &  =t_{2k}r_{2k-1}e^{-r_{2k-1}}=t_{0}r_{2k-1}e^{-r_{2k-1}}%
=g_{0}(r_{2k-1},t_{0})
\end{align*}

Therefore,%
\[
r_{2k+1}=g_{1}(g_{0}(r_{2k-1},t_{0}),t_{0})=g_{1}\circ g_{0}(r_{2k-1},t_{0})
\]

A similar calculation shows that%
\[
r_{2k+2}=g_{0}(g_{1}(r_{2k},t_{0}),t_{0})=g_{0}\circ g_{1}(r_{2k},t_{0})
\]
and the proof of (a) is complete.

(b) Note that $g(\bar{r})=\bar{r}e^{d/2-\bar{r}}=\bar{r}$ so the trace of $g$
contains $(\bar{r},\bar{r}).$ The curves $g_{0},g_{1}$ coincide if
$t_{0}=e^{d}/t_{0}$, i.e. $t_{0}=e^{d/2}.$ This happens if the initial values
\thinspace$r_{-1},r_{0}$ satisfy (\ref{21}). In this case, $(r_{-1},r_{0})$ is
clearly on the trace of $g$ and by (\ref{star2})%
\[
r_{1}=t_{1}r_{0}e^{-r_{0}}=\frac{e^{d}}{t_{0}}r_{0}e^{-r_{0}}=t_{0}%
r_{0}e^{-r_{0}}=g(r_{0})
\]

Therefore, the point $(r_{0},r_{1})$ is also on the trace of $g.$ Since
$t_{n}=t_{0}$ for all $n$ if $t_{0}=e^{d/2}$ the same argument applies to
$(r_{n},r_{n+1})$ for all $n$ and completes the proof by induction.
\end{proof}

\medskip

Note that the invariant curve $g$ does not depend on initial values. There is
also the following useful fact about $g$.

\begin{lemma}
\label{p3}The mapping $g$ has a period-three point for $d\geq6.26.$
\end{lemma}

\begin{proof}
Let $a=d/2.$ The third iterate of $g$ is%
\[
g^{3}(r)=r\exp(3a-r-2re^{a-r}+e^{a-re^{a-r}})
\]

In particular,
\[
g^{3}(1)<\exp(3a-1-e^{a-1})\doteq h(a)
\]

Solving $h(a)=1$ numerically yields the estimate $a\approx3.12.$ Since $h(a)$
is decreasing if $a>2.1$ it follows that $h(a)<1$ if $a\geq3.13$. Therefore,
$g^{3}(1)<1$ for $d\geq6.26.$ Further, for $\varepsilon\in(0,a)$
\begin{align*}
g^{3}(a-\varepsilon)  &  >(a-\varepsilon)\exp\left[  2a+\varepsilon
-2(a-\varepsilon)e^{\varepsilon}+e^{a(1-e^{\varepsilon})}\right] \\
&  >(a-\varepsilon)\exp[e^{-a(e^{\varepsilon}-1)}-2a(e^{\varepsilon}-1)]
\end{align*}

For sufficiently small $\varepsilon$ the exponent is positive so we may assert
that%
\[
g^{3}(1)<1<a-\varepsilon<g^{3}(a-\varepsilon)
\]

Hence, there is a root of $g^{3}(r)$, or a period-three point of $g$ in the
interval $(1,a)$ if $a\geq3.13$, i.e. $d\geq6.26$.
\end{proof}

The function compositions in Lemma \ref{spt} are specifically the following
mappings:%
\begin{align*}
g_{1}\circ g_{0}(r,t_{0})  &  =re^{d-r-t_{0}re^{-r}},\\
g_{0}\circ g_{1}(r,t_{0})  &  =re^{d-r-t_{1}re^{-r}},\quad t_{1}=\frac{e^{d}%
}{t_{0}}.
\end{align*}

To simplify our notation, for each $t\in(0,\infty)$ define the class of
functions $f_{t}:(0,\infty)\rightarrow(0,\infty)$ as%
\[
f_{t}(r)=re^{d-r-tre^{-r}}.
\]

We also abbreviate $f_{t_{0}}$ as $f_{0}$, $f_{t_{1}}$ as $f_{1}$,
$g_{0}(\cdot,t_{0})$ as $g_{0}$ and $g_{1}(\cdot,t_{0})$ as $g_{1}$. Then we
see from the preceding discussion that
\begin{equation}
g_{1}\circ g_{0}=f_{0},\quad g_{0}\circ g_{1}=f_{1}. \label{ggf}%
\end{equation}

According to Lemma \ref{spt}, iterations of $f_{0}$ generate the odd-indexed
terms of a solution of (\ref{rmsa}) and iterations of $f_{1}$ generate the
even-indexed terms.

The next result furnishes a relationship between $f_{i}$ and $g_{i}$ for
$i=0,1.$

\begin{lemma}
\label{sc}Let $t_{0}\in(0,\infty)$ be fixed and $t_{1}=e^{d}/t_{0}.$ Then%
\begin{equation}
f_{1}\circ g_{0}=g_{0}\circ f_{0}\quad\text{and\quad}f_{0}\circ g_{1}%
=g_{1}\circ f_{1}. \label{sce}%
\end{equation}

\end{lemma}

\begin{proof}
This may be established by straightforward calculation using the definitions
of the various functions, or alternatively, use (\ref{ggf}) to obtain
\[
f_{1}\circ g_{0}=\left(  g_{0}\circ g_{1}\right)  \circ g_{0}=g_{0}\circ
(g_{1}\circ g_{0})=g_{0}\circ f_{0}%
\]

This proves the first equality in (\ref{sce}) and the second equality is
proved similarly.
\end{proof}

\medskip

The equalities in (\ref{sce}) are not conjugacies since $g_{0}$ and $g_{1}$
are not one-to-one. However, the following is implied.

\begin{lemma}
\label{cyc}(a) If $\{s_{1},s_{2},\ldots,s_{q}\}$ is a $q$-cycle of $f_{0},$
i.e. a solution (listed in the order of iteration) of
\begin{equation}
s_{n+1}=f_{0}(s_{n})=s_{n}e^{d-s_{n}-t_{0}s_{n}e^{-s_{n}}} \label{1sc}%
\end{equation}
with minimal (or prime) period $q\geq1$ then $\{g_{0}(s_{1}),g_{0}%
(s_{2}),\ldots,g_{0}(s_{q})\}$ is a $q$-cycle of $f_{1},$ i.e. a solution of
\begin{equation}
u_{n+1}=f_{1}(u_{n})=u_{n}e^{d-u_{n}-t_{1}u_{n}e^{-u_{n}}} \label{2sc}%
\end{equation}
with period $q$ (listed in the order of iteration). Similarly, if
$\{u_{1},u_{2},\ldots,u_{q}\}$ is a $q$-cycle of $f_{1},$ i.e. a solution of
(\ref{2sc}) with minimal period $q\geq1$ then $\{g_{1}(u_{1}),g_{1}%
(u_{2}),\ldots,g_{1}(u_{q})\}$ is a $q$-cycle of $f_{0},$ i.e. solution of
(\ref{1sc}) with period $q$.

(b) If $\{s_{n}\}$ is a non-periodic solution of (\ref{1sc}) then $\{g_{0}%
(s${}$_{n})\}$ is a non-periodic solution of (\ref{2sc}). Similarly, if
$\{u_{n}\}$ is a non-periodic solution of (\ref{2sc}) then $\{g_{1}(u_{n})\}$
is a non-periodic solution of (\ref{1sc}).
\end{lemma}

\begin{proof}
(a) By the hypothesis, $f_{0}(s_{n+q})=s_{n}$ for all $n$ and in the order of
iteration
\[
f_{0}(s_{k})=s_{k+1}\quad\text{for }k=1,\ldots,q-1\text{\quad and\quad}%
f_{0}(s_{q})=s_{1}.
\]

By Lemma \ref{sc},
\[
f_{1}(g_{0}(s_{n+q}))=g_{0}(f_{0}(s_{n+q}))=g_{0}(s_{n})
\]
and also
\begin{align*}
f_{1}(g_{0}(s_{k}))  &  =g_{0}(f_{0}(s_{k}))=g_{0}(s_{k+1})\quad\text{for
}k=1,\ldots,q-1,\\
f_{1}(g_{0}(s_{q}))  &  =g_{0}(f_{0}(s_{q}))=g_{0}(s_{1})
\end{align*}

It follows that $\{g_{0}(s_{1}),g_{0}(s_{2}),\ldots,g_{0}(s_{q})\}$ is a
periodic solution of (\ref{2sc}) with period $q$, listed in the order of
iteration. The rest of (a) is proved similarly.

(b) Let $\{s_{n}\}$ be a solution of (\ref{1sc}) such that $\{g_{0}(s${}%
$_{n})\}$ is a periodic solution of (\ref{2sc}). Then $\{g_{1}(g_{0}(s${}%
$_{n}))\}$ is a periodic solution of (\ref{1sc}) by (a). Since $g_{1}(g_{0}%
(s${}$_{n}))=f_{0}(s_{n})$ by (\ref{ggf}) we may conclude that there is a
positive integer $q$ such that $f_{0}^{q}(s_{n})=f_{0}(s_{n})=s_{n+1}$ for all
$n.$ Thus $s_{n+1}=f_{0}^{q-1}(s_{n+1})$ for all $n$ and it follows that
$\{s_{n}\}$ is a periodic solution of (\ref{1sc}). This proves the first
assertion in (b); the second assertion is proved similarly.
\end{proof}

\medskip

The next result concerns the local stability of the periodic solutions of
(\ref{1sc}) and (\ref{2sc}).

\begin{lemma}
\label{cya}If $\{s_{1},s_{2},\ldots,s_{q}\}$ is a periodic solution of
(\ref{1sc}) with minimal period $q$ such that $s_{k}\not =1$ for
$k=1,2,\ldots,q$ and
\begin{equation}
\prod\limits_{k=1}^{q}f_{0}^{\,\prime}(s_{k})<1 \label{las}%
\end{equation}
then $\{g_{0}(s_{1}),\ldots,g_{0}(s_{q})\}$ is a solution of (\ref{2sc}) of
period $q$ with $\prod\limits_{k=1}^{q}f_{1}^{\,\prime}(g_{0}(s_{k}))<1.$
Similarly, if $\{u_{1},u_{2},\ldots,u_{q}\}$ is a periodic solution of
(\ref{2sc}) with $u_{k}\not =1$ for $k=1,2,\ldots,q$ and
\[
\prod\limits_{k=1}^{q}f_{1}^{\,\prime}(u_{k})<1
\]
then $\{g_{1}(u_{1}),g_{1}(u_{2}),\ldots,g_{1}(u_{q})\}$ is a solution of
(\ref{1sc}) of period $q$ with $\prod\limits_{k=1}^{q}f_{0}^{\,\prime}%
(g_{1}(u_{k}))<1.$
\end{lemma}

\begin{proof}
By Lemma \ref{sc} and the chain rule
\[
f_{1}^{\prime}(g_{0}(r))g_{0}^{\prime}(r)=g_{0}^{\prime}(f_{0}(r))f_{0}%
^{\prime}(r)
\]

Now $g_{0}^{\prime}(r)=(1-r)t_{0}e^{-r}\not =0$ if $r\not =1$. Thus if
$s_{k}\not =1$ for $k=1,2,\ldots,q$ then%
\begin{align*}
\prod\limits_{k=1}^{q}f_{1}^{\,\prime}(g_{0}(s_{k}))  &  =\frac{g_{0}^{\prime
}(f_{0}(s_{1}))f_{0}^{\prime}(s_{1})}{g_{0}^{\prime}(s_{1})}\frac
{g_{0}^{\prime}(f_{0}(s_{2}))f_{0}^{\prime}(s_{2})}{g_{0}^{\prime}(s_{2}%
)}\cdots\frac{g_{0}^{\prime}(f_{0}(s_{q}))f_{0}^{\prime}(s_{q})}{g_{0}%
^{\prime}(s_{q})}\\
&  =\frac{g_{0}^{\prime}(s_{2})f_{0}^{\prime}(s_{1})}{g_{0}^{\prime}(s_{1}%
)}\frac{g_{0}^{\prime}(s_{3})f_{0}^{\prime}(s_{2})}{g_{0}^{\prime}(s_{2}%
)}\cdots\frac{g_{0}^{\prime}(s_{1})f_{0}^{\prime}(s_{q})}{g_{0}^{\prime}%
(s_{q})}\\
&  =\prod\limits_{k=1}^{q}f_{0}^{\,\prime}(s_{k})<1
\end{align*}

The second assertion is proved similarly.
\end{proof}

\medskip

We are now ready to explain some of what appears in Figure \ref{one}.

\begin{theorem}
\label{mus}Let $d>0.$

(a) Except among solutions whose initial values satisfy (\ref{21}) there are
no positive solutions of (\ref{rmsa}) that are periodic with an odd period.

(b) If $d\geq6.26$ then (\ref{rmsa}) has periodic solutions of all possible
periods, including odd periods, as well as chaotic solutions in the sense of
Li and Yorke.

(c) Let $r_{-1},r_{0}>0$ be given initial values and define $t_{0}$ by
(\ref{star1}). Assume that $t_{0}\not =e^{d/2}$ and the sequence of iterates
$\{f_{0}^{n}(r_{-1})\}$ of the map $f_{0}$ converges to a minimal $q$-cycle
$\{s_{1},\ldots,s_{q}\}$. Then the corresponding solution $\{r_{n}\}$ of
(\ref{rmsa}) converges to the cycle $\{s_{1},g_{0}(s_{1}),\ldots,s_{q}%
,g_{0}(s_{q})\}$ of minimal period $2q$ in the sense that
\begin{equation}
\lim_{k\rightarrow\infty}|r_{2(k+j)-1}-s_{j}|=\lim_{k\rightarrow\infty
}|r_{2(k+j)}-g_{0}(s_{j})|=0\quad\text{for}\quad j=1,\ldots,q \label{lims}%
\end{equation}

(d) If $\{s_{1},\ldots,s_{q}\}$ in (c) satisfies (\ref{las}) and $s_{j}%
\not =1$ for $j=1,\ldots,q$ then for intial values $r_{-1}^{\prime}>0$ and
$r_{0}^{\prime}=g_{0}(r_{-1}^{\prime})$ where $|r_{-1}^{\prime}-r_{-1}|$ is
sufficiently small, the sequence $\{f_{0}^{n}(r_{-1}^{\prime})\}$ converges to
$\{s_{1},\ldots,s_{q}\}$ and (\ref{lims}) holds$.$

(e) Let $r_{-1},r_{0}>0$ be given initial values and define $t_{0}$ by
(\ref{star1}). If the sequence of iterates $\{f_{0}^{n}(r_{-1})\}$ of the map
$f_{0}$ is non-periodic then (\ref{rmsa}) has a non-periodic solution.
\end{theorem}

\begin{proof}
(a) This statement is an immediate consequence of Lemma \ref{spt} since the
number of points in a cycle must divide two, i.e. the number of curves
$g_{0},g_{1}$. An exception occurs when (\ref{21}) holds and the curves
$g_{0},g_{1}$ coincide.

(b) Suppose that the initial values $r_{-1},r_{0}$ satisfy (\ref{21}). Then
$g_{0}=g_{1}=g$ and the trace of $g$ contains the orbits of (\ref{rmsa}) since
the trace of $g$ is invariant by Lemma \ref{spt}. By Lemma \ref{p3} $g$ has a
period-three point if $d\geq6.24$ and in this case, (\ref{rmsa}) has solutions
with all possible periods in the state-space, including odd periods. In
addition, iterates of $g$ also exhibit chaos in the sense of \cite{LY}. For
(\ref{rmsa}) this is manifested in the state-space on the trace of $g$ if the
initial point $(r_{-1},r_{0})$ is on the trace of $g$. For arbitrary initial
values, odd periods do not occur by (a) and chaotic behavior generally occurs
on the pair of curves $g_{0},g_{1}$; see the Remark following this proof.

(c) This is an immediate consequence of Lemmas \ref{spt} and \ref{cyc}.

(d) If $|r_{-1}^{\prime}-r_{-1}|$ is sufficiently small then Lemma \ref{cya}
implies that the sequence $\{f_{0}^{n}(r_{-1}^{\prime})\}$ converges to
$\{s_{1},\ldots,s_{q}\}$. Now, if $r_{0}^{\prime}=g_{0}(r_{-1}^{\prime})$ then
$r_{0}^{\prime}/r_{-1}^{\prime}e^{r_{-1}^{\prime}}=t_{0}$ and thus,
(\ref{lims}) holds by Part (c).

(e) This is clear from Lemmas \ref{spt} and \ref{cyc}.
\end{proof}

\medskip

\begin{remark}
1. Theorem \ref{mus} explains how qualitatively different solutions in Figure
\ref{one} are generated by different initial values. Changes in the initial
value $r_{0}$ of (\ref{rmsa}) while $r_{-1}$ is fixed result, by (\ref{star1})
in changes in the parameter value $t_{0}$ in the mapping $f_{0}$. The
one-dimensional map $f_{0}$ generates different types of orbits with different
values of $t_{0}$ in the conventional way that is explained by the basic
theory. All of these orbits, combined with the iterates of the shadow map
$f_{1}$ appear in the state-space of (\ref{rmsa}) as points on the
aforementioned pair of curves.

2. Part (d) of Theorem \ref{mus} explains the sense in which the solutions of
(\ref{rmsa}) are stable and therefore appear as shown in Figure \ref{one}.
This is not local or linearized stability since if $r_{0}^{\prime}\not =%
g_{0}(r_{-1}^{\prime})$ then
\[
t_{0}^{\prime}=\frac{r_{0}^{\prime}}{r_{-1}^{\prime}e^{-r_{-1}^{\prime}}%
}\not =t_{0}%
\]
and with the different parameter value $t_{0}^{\prime}$, $\{f_{0}^{n}%
(r_{-1}^{\prime})\}$ may not converge to $\{s_{1},\ldots,s_{q}\}$ even if
$|r_{-1}^{\prime}-r_{-1}|$ is small enough to imply local convergence for the
iterates of $f_{0}$ defined with the original value $t_{0}$.

3. In Parts (a) and (b) of Theorem \ref{mus} if the initial point is not on
the trace of $g$ then the occurrence of all possible even periods and chaotic
behavior is observed for smaller values of $d.$ In fact, since $g$ involves
$d/2$ but $f_{0}$ involves $d$ it follows that $f_{0}$ actually has period 3
points for $d\geq3.13$ if the initial values yield a sufficiently small value
of $t_{0}.$ In Figure \ref{two} a stable three-cycle is identified for $d=3.6$
and initial values satisfying $r_{0}=r_{-1}e^{-r_{-1}}$ (so that $t_{0}=1$).
Odd periods do not occur for (\ref{rmsa}) in this case but all possible even
periods, as well as chaotic behavior (on curve-pairs) do occur.
\end{remark}

\begin{figure}[ptbh]
\centering
\includegraphics[width=2.45in,height=2.1in,keepaspectratio]{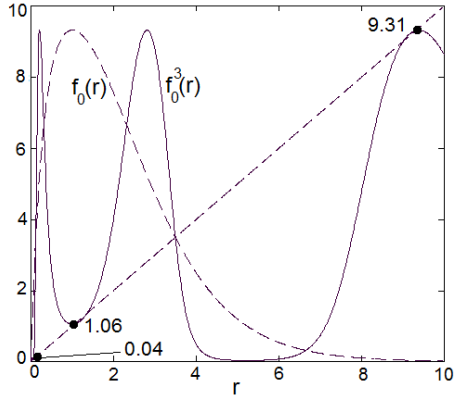}
\caption{Occurrence of period 3 for the associated interval map}%
\label{two}%
\end{figure}

\subsection{Further results: convergence to two-cycles}

The preceding results indicate that the solutions of (\ref{1sc}) and
(\ref{2sc}) determine the solutions of (\ref{rmsa}). From Theorem \ref{mus} it
is evident that complex behavior tends to occur when $d$ is sufficiently
large. Otherwise, the solutions of (\ref{rmsa}) tend to behave more simply as
noted in Theorem \ref{gas}. We now consider the occurrence of two-cycles for a
range of values of $d$ that are not too large but extend the range in Theorem
\ref{gas}(b), by examining the following first-order difference equation that
is derived from (\ref{1sc}) and (\ref{2sc})%
\begin{equation}
r_{n+1}=r_{n}e^{d-r_{n}-\gamma r_{n}e^{-r_{n}}},\quad\gamma>0 \label{dr}%
\end{equation}

\begin{lemma}
If $0<d\leq2$ then (\ref{dr}) has a unique positive fixed point $\bar{x}$.
\end{lemma}

\begin{proof}
Existence: Let $\eta(x)=d-x-\gamma xe^{-x}$. The nonzero fixed points of
(\ref{dr}) must satisfy $e^{\eta(x)}=1$, i.e. $\eta(x)=0$. Since $\eta(0)=d>0$
and $\eta(d)=-\gamma de^{-d}<0$ there is a real number $\bar{x}\in(0,a)$ such
that $\eta(\bar{x})=0$. This proves existence.

Uniqueness: Note that $\eta^{\prime}(x)=-1-\gamma e^{-x}+\gamma xe^{-x}$.

\textit{Case 1}: $\gamma\leq e$; The function $xe^{-x}$ is maximized on
$(0,\infty)$ at $h(1)=e^{-1}$ so%
\[
\eta^{\prime}(x)=-1-\gamma e^{-x}+\gamma xe^{-x}\leq-1+1-\gamma e^{-x}=-\gamma
e^{-x}<0
\]

It follows that $\eta(x)$ is decreasing on $(0,\infty)$ for this case and has
a unique zero that occurs at $\bar{x}$.

\textit{Case 2}: $e<\gamma<e^{2}$; Consider the function $p(x)=x+\gamma
xe^{-x}$. Now%
\[
p^{\prime}(x)=1+\gamma e^{-x}-\gamma xe^{-x}=e^{-x}(e^{x}+\gamma-\gamma x)
\]

The function $q(x)=e^{x}+\gamma-\gamma x$ attains a minimum value at
$x=\ln(\gamma)$ since $q^{\prime}(x)=e^{x}-\gamma$. Furthermore,
\[
q(\ln(\gamma))=2\gamma-\gamma\ln(\gamma)=\gamma(2-\ln(\gamma))>0
\]
for $\gamma<e^{2}$. This implies that $p^{\prime}(x)>0$ on $(0,\infty)$ and
therefore $p(x)$ is increasing on $(0,\infty)$. Since $\eta(x)=d-p(x)$, this
implies that $\eta(x)$ is decreasing on $(0,\infty)$ and therefore it has a
unique zero that occurs at $\bar{x}$.

\textit{Case 3}: $\gamma>e^{2}$; In this case, we know that $\eta(x)$ is
decreasing on $[0,1]$ and $\eta(x)<0$ for $x\in\lbrack d,\infty)$. Thus it
remains to establish that $\eta(x)<0$ on $(1,d)$.%
\[
\eta(x)=d-x-\gamma xe^{-x}<d-1-e^{2-x}<d-2\leq0
\]

Thus $\eta(x)$ has a unique zero that occurs at $\bar{x}$ and this completes
the proof for all the above cases.
\end{proof}

\medskip

The above observations also indicate that $\eta(x)>0$ for $x\in(0, \bar{x})$
and $\eta(x)<0$ for $x\in(\bar{x}, \infty)$, which we will use in our further
analysis. Before examining the stability profile of $\bar{x}$, we need to
explore the properties of the function $f(x)$.

Since $f(x)=xe^{d-x-\gamma xe^{-x}}=xe^{\eta(x)}$, then $f^{\prime}%
(x)=e^{\eta(x)}+x\eta^{\prime}(x)e^{\eta(x)}$. By direct calculations,
$f^{\prime}(x)$ can be written as%
\[
f^{\prime}(x)=e^{\eta(x)}(1-x)(1-\gamma xe^{-x})
\]

It follows that $f$ has critical points when $x=1$ and $1-\gamma xe^{-x}=0$.
Now we consider the function $\phi(x)=1-\gamma xe^{-x}$, which has a critical
point at $x=1$, since $\phi^{\prime}(x)=\gamma e^{-x}(1-x)$. Hence it is
decreasing on $(0,1)$ and increasing on $(1, \infty)$ and $\phi(1)=1-\frac
{\gamma}{e}$ is the minimum of the function.

\medskip(i) When $\gamma<e$, then $\phi(1)>0$, so $\phi(x)>0$ on $(0, \infty
)$, hence $f(x)$ has only one critical point at $x=1$. When $\gamma=e,
\phi(1)=0$, and again, the only critical point of $f(x)$ occurs at $x=1$. We
further break down the case of $\gamma\leq e$ into the following subcases:

\medskip a. When $d<1+\frac{\gamma}{e}$, $\eta(1)=d-1-\frac{\gamma}{e}<0$,
thus $\bar{x}<1$. Moreover, $f(1)=d-1-\frac{\gamma}{e}<1$, which lets us
conclude that $f(x)<1$ for all $x\in(0, \infty)$.

\medskip b. When $d\geq1+\frac{\gamma}{e}$, $\eta(1)=d-1-\frac{\gamma}{e}%
\geq0$. This implies that $\bar{x}>1$ if $d> 1+\frac{\gamma}{e}$ and $\bar
{x}=1$ if $d= 1+\frac{\gamma}{e}$.

\medskip(ii) When $\gamma>e, \phi(1)<0$, so $f(x)$ has three critical points
at $x^{\prime}<1, x^{\prime}=1, x^{\prime\prime}>1$.

On $(0,x^{\prime})$, $1-x>0$ and $\phi(x)>0$, so $f$ is increasing. On
$(x^{\prime},1)$, $1-x>0$ and $\phi(x)<0$, so $f$ is decreasing. On
$(1,x^{\prime\prime})$, $1-x<0$ and $\phi(x)<0$, so $f$ is increasing. On
$(x^{\prime\prime},\infty)$, $1-x<0$ and $\phi(x)>0$, so $f$ is decreasing. By
the above observations, it follows that $x^{\prime},x^{\prime\prime}$ are
local maxima and $1$ is a minimum point. Next observe that%
\[
f(1)=e^{2-1-\frac{\gamma}{e}}<1
\]

Given that $\gamma x^{\prime}e^{-x^{\prime}}=\gamma x^{\prime\prime
}e^{-x^{\prime\prime}}=1$,
\[
f(x^{\prime})=x^{\prime}e^{d-x^{\prime}-\gamma x^{\prime}e^{-x^{\prime}}%
}=x^{\prime}e^{d-x^{\prime}-1}<x^{\prime}e^{2-x^{\prime}-1}=x^{\prime
}e^{1-x^{\prime}}
\]

Similarly, $f(x^{\prime\prime})<x^{\prime\prime}e^{1-x^{\prime\prime}}$. Now,
the function $s(x)=xe^{1-x}$ attains its maximum at $x=1$, since $s^{\prime
}(x)=(1-x)e^{1-x}$. Since $s(1)=1$, this implies that $s(x)<1$ for all
$x\neq1,x>0$. This implies that $f(x^{\prime}),f(x^{\prime\prime})<1$ as well,
thus for this case $f(x)<1$ for all $x\in(0,\infty).$

\medskip

Now we establish the global stability of $\bar{x}$.

\begin{lemma}
\label{gc1} If $0<d\leq2$ then every solution to (\ref{dr}) from positive
initial values converges to $\bar{x}$.
\end{lemma}

\begin{proof}
We establish convergence to $\bar{x}$ by showing that $|f(x)-\bar{x}%
|<|x-\bar{x}|$ for $x\neq\bar{x}$. This is equivalent to%
\begin{subequations}
\begin{align}
x  &  <f(x)<2\bar{x}-x\;\;\text{for}\;\;x<\bar{x}\label{gas1}\\
x  &  >f(x)>2\bar{x}-x\;\;\text{for}\;\;x>\bar{x} \label{gas2}%
\end{align}

The first inequalities in (\ref{gas1}-\ref{gas2}) are straightforward to
establish: since $\eta(x)>0$ for $x<\bar{x}$ and $\eta(x)<0$ for $x>\bar{x}$,
then $f(x)=xe^{\eta(x)}>x$ if $x<\bar{x}$ and $f(x)=xe^{\eta(x)}<x$ if
$x>\bar{x}$.

To establish the second inequalities in (\ref{gas1})-(\ref{gas2}), let
\end{subequations}
\[
t(x)=f(x)+x-2\bar{x}%
\]

Notice that $t(0)=-2\bar{x}<0$ and $t(\bar{x})=0$. We now show that the
inequalities $f(x)<2\bar{x}-x$ for $x<\bar{x}$ and $f(x)>2\bar{x}-x$ for
$x>\bar{x}$ are equivalent to $t(x)<0$ for $x<\bar{x}$ and $t(x)>0$ for
$x>\bar{x}$, respectively. We establish this by showing that $t(x)$ is
strictly increasing on $(0,\infty)$, i.e.%
\[
t^{\prime}(x)=f^{\prime}(x)+1>0\;\;\text{for}\;\;x>0
\]

We establish the above result by considering two cases: \medskip\textit{Case
1}: $\gamma\leq e$; recall that $f(x)$ is maximized at the unique critical
point $x=1$. Thus $f^{\prime}(x)>0$ for $x<1$ and $f^{\prime}(x)<0$ for $x>1$.
We also showed that $1-\gamma xe^{-x}>0$ for $x>0$. Thus for all $x>1$, since
$d\leq2$%
\begin{align}
|f^{\prime}(x)|  &  \leq e^{2-x-\gamma xe^{-x}}(x-1)(1-\gamma xe^{-x}%
)\nonumber\\
&  =(x-1)e^{1-x}e^{1-\gamma xe^{-x}}(1-\gamma xe^{-x})\nonumber\\
&  <e^{-1}e^{1-\gamma xe^{-x}}(1-\gamma xe^{-x})\nonumber\\
&  =e^{-\gamma xe^{-x}}(1-\gamma xe^{-x})<1\nonumber
\end{align}
i.e. $t^{\prime}(x)>0$ for $x>0$ and inequalities in (\ref{gas1})-(\ref{gas2}) follow.

\medskip\textit{Case 2}: $\gamma>e$; in this case, $f(x)$ has three critical
points occurring at $x^{\prime}<1$, $1$ and $x^{\prime\prime}>1$, where
$x^{\prime}$ and $x^{\prime\prime}$ are maxima and $1$ is a minimum. Thus%
\begin{align*}
f^{\prime}(x)  &  >0\;\;\text{and}\;\;1-\gamma xe^{-x}>0\;\;\text{for}%
\;\;x\in(0,x^{\prime})\\
f^{\prime}(x)  &  <0\;\;\text{and}\;\;1-\gamma xe^{-x}<0\;\;\text{for}%
\;\;x\in(x^{\prime},1)\\
f^{\prime}(x)  &  >0\;\;\text{and}\;\;1-\gamma xe^{-x}<0\;\;\text{for}%
\;\;x\in(1,x^{\prime\prime})\\
f^{\prime}(x)  &  <0\;\;\text{and}\;\;1-\gamma xe^{-x}>0\;\;\text{for}%
\;\;x\in(x^{\prime\prime},\infty)
\end{align*}

Thus $f^{\prime}(x)<0$ if either $x<1$ and $1-\gamma xe^{-x}<0$ or $x>1$ and
$1-\gamma xe^{-x}>0$. If $x<1$ and $1-\gamma xe^{-x}<0$, then%
\begin{align}
|f^{\prime}(x)|  &  \leq e^{2-x-\gamma xe^{-x}}(1-x)(\gamma xe^{-x}%
-1)\nonumber\\
&  =(\gamma xe^{-x}-1)e^{1-\gamma xe^{-x}}e^{1-x}(1-x)\nonumber\\
&  <e^{-1}e^{1-x}(1-x)\nonumber\\
&  =e^{-x}(1-x)<1\nonumber
\end{align}

If $x>1$ and $1-\gamma xe^{-x}>0$, then%
\begin{align}
|f^{\prime}(x)|  &  \leq e^{2-x-\gamma xe^{-x}}(x-1)(1-\gamma xe^{-x}%
)\nonumber\\
&  =(x-1)e^{1-x}(1-\gamma xe^{-x})e^{1-\gamma xe^{-x}}\nonumber\\
&  <e^{-1}e^{1-\gamma xe^{-x}}(1-\gamma xe^{-x})\nonumber\\
&  =e^{-\gamma xe^{-x}}(1-\gamma xe^{-x})<1\nonumber
\end{align}

In either case, if $f(x)$ is decreasing then $-1<f^{\prime}(x)<0$, thus
$t^{\prime}(x)=f^{\prime}(x)+1>0$, thus $t(x)$ is increasing for $x>0$, from
which the second inequalities in (\ref{gas1})-(\ref{gas2}) follow.
\end{proof}

\medskip

By Lemmas \ref{spt} and \ref{gc1}, the even and odd terms of (\ref{rmsa})
converge to $M=\bar{x}_{t_{0}}>0$ and $m=\bar{x}_{t_{1}}>0$, proving the
existence and stability of a two-cycle in the sense described in Theorem
\ref{mus}(c). Since $M$ and $m$ must satisfy%

\[
m=Me^{d-M-m}\;\;\text{and}\;\;M=me^{d-m-M}%
\]
and%
\[
Mm=mMe^{2d-2(M+m)}\;\;\text{i.e.}\;\;\;e^{2d-2(M+m)}=1
\]
we conclude that $M+m=d$. Thus the following extension of Theorem \ref{gas}(b)
is obtained.

\begin{theorem}
Let $0<d\leq2.$ Then every non-constant, positive solution of (\ref{rmsa})
converges, in the sense of Theorem \ref{mus}(c), to a two-cycle $\{\rho
_{1},\rho_{2}\}$ that satisfy $\rho_{1}+\rho_{2}=d,$ i.e. the mean value of
the limit cycle is the fixed point $\bar{r}=d/2.$
\end{theorem}

As previously mentioned, (\ref{rmsa}) is valid when $c_{1,n}>0$ has period 2.
In this case, the solution of (\ref{o2s}) corresponding to $\{r_{n}\}$ of
(\ref{rmsa}) is $x_{n}=r_{n}/c_{1,n}$ which also converges to a sequence of
period 2. Thus we have the following corollary.

\begin{corollary}
Assume in the system (\ref{rm1})-(\ref{rm2}) that $\sigma_{1,n}=\sigma_{1}$,
$\alpha_{n}=\alpha$, $\beta_{n}=\beta$ are positive constants and
$c_{2,n}=\sigma_{1}c_{1,n}$ for all $n$ where $c_{1,n}$ has period two with
$c_{1,2k-1}=\xi_{1}$ and $c_{1,2k}=\xi_{2}$ where $\xi_{1},\xi_{2}>0$.

(a) If $\alpha+\ln(\sigma_{1}\beta)\in(0,2]$ then every orbit $\{(x_{n}%
,y_{n})\}$\ is determined as%
\[
x_{n}=\frac{r_{n}}{c_{1,n}},\quad y_{n}=\frac{r_{n+1}}{\sigma_{1}c_{1,n+1}}.
\]

(b) Every orbit in the positive quadrant converges to a two-cycle%
\[
\left\{  \left(  \frac{\rho_{1}}{\xi_{1}},\frac{\rho_{2}}{\sigma_{1}\xi_{2}%
}\right)  ,\left(  \frac{\rho_{2}}{\xi_{2}},\frac{\rho_{1}}{\sigma_{1}\xi_{1}%
}\right)  \right\}
\]
where $\rho_{i}=\lim_{k\rightarrow\infty}r_{2k-i}$ for $i=1,2$ and $\rho
_{1}+\rho_{2}=\alpha+\ln(\sigma_{1}\beta).$
\end{corollary}

\subsection{A concluding remark on multistability}

We finally mention a feature of (\ref{rmsa}) that may make its multistable
nature less surprising. Consider the following class of nonautonomous
first-order equations%
\[
x_{n+1}=x_{n}e^{\gamma_{n}-\theta_{n}x_{n}}%
\]
where $\gamma_{n},\theta_{n}$ are given sequences of period 2 with $\theta
_{n}>0$ for all $n$. The change of variable $u_{n}=\theta_{n}x_{n}$ reduces
this equation to%
\begin{equation}
u_{n+1}=u_{n}e^{c_{n}-u_{n}},\quad c_{n}=\gamma_{n}+\ln\frac{\theta_{n+1}%
}{\theta_{n}} \label{eo2}%
\end{equation}

This equation can be written as%
\[
u_{n+1}=u_{n-1}e^{c_{n-1}-u_{n-1}}=u_{n-1}e^{c_{n-1}+c_{n}-u_{n-1}-u_{n}}%
\]

Since $c_{n}$ has period 2, the sum $c_{n-1}+c_{n}=d$ is a constant and
(\ref{rmsa}) is obtained.

If $r_{-1}=u_{0}$ and $r_{0}=u_{1}=u_{0}e^{c_{0}-u_{0}}$ then the
corresponding solution of (\ref{rmsa}) is the solution of (\ref{eo2}) with the
arbitrary initial value $u_{0}.$ Therefore, all solutions of (\ref{eo2})
appear among the solutions of (\ref{rmsa}) but not conversely. In fact, if
$c_{n}^{\prime}$ is any other sequence of period 2 such that $c_{n}^{\prime
}+c_{n-1}^{\prime}=d$ then while
\[
u_{n+1}=u_{n}e^{c_{n}^{\prime}-u_{n}}%
\]
is a different equation than (\ref{eo2}), it yields exactly the same
second-order equation (\ref{rmsa}). Hence, the following assertion is justified:

\begin{proposition}
The solutions of (\ref{rmsa}) include the solutions of all first-order
equations of type (\ref{eo2}) with $c_{n}+c_{n-1}=d.$
\end{proposition}

The coexistence of solutions of so many different first-order equations among
the solutions of (\ref{rmsa}) is a further indication of the diversity of
solutions that the latter may exhibit.

\section{Conclusion and future directions}

In this paper we examine the dynamics of the non-autonomous system
(\ref{pc1})-(\ref{pc2}) whose special cases appear in stage-structured models
of populations that are of Ricker type, or overcompensatory. In Section
\ref{exti} we obtain conditions that imply uniform boundedness as well as
global convergence to zero with variable parameters. In biological models
these results give general conditions for the species' extinction. We have
also shown that in periodic environments certain stocking strategies do not
prevent extinction.

In Section \ref{cr} we study the dynamics of a special case of the system that
is mathematically interesting. We use semiconjugate factorization to show that
in a wider range of parameters than what is considered in \cite{FHL} complex
and multistable behavior occurs. 

The results in Section \ref{cr} concern Equation (\ref{rmsa}) which is
autonomous (even if the system is not). For future investigation one may
consider the more general, non-autonomous equation (\ref{rms}) with periodic
$d_{n}.$ Preliminary work on this periodic case shows that the dynamics of
(\ref{rms}) where $d_{n}$ has an odd period (including the autonomous case
$p=1$) is substantially and qualitatively different from the case where
$d_{n}$ has an even period. 

Another generalization of (\ref{rmsa}), namely the autonomous equation%
\begin{equation}
r_{n+1}=r_{n-1}e^{d-br_{n-1}-cr_{n}}\label{rnms}%
\end{equation}
where $b,c>0$ exhibits different dynamics than (\ref{rmsa}) when $b\not =c.$
In particular, we expect that mulitstable orbits will not occur although
complex behavior is possible. There is currently no comprehensive study of the
dynamics of (\ref{rnms}) that we are aware of so obtaining significant details
on the dynamics of this equation would be desirable.

\end{document}